\documentclass[runningheads,11pt]{article}

\usepackage{fixltx2e,breakurl}
\usepackage{latexsym,amsmath,amssymb,amsfonts,amsthm}
\usepackage{graphicx,epsfig}
\usepackage{url}
\usepackage{multirow}
\usepackage[nocompress]{cite}
\usepackage[scale=0.8]{geometry}
\usepackage{shuffle}
\usepackage{tikzsymbols}
\usepackage{tabularx}

\renewcommand{\epsilon}{\varepsilon}

\theoremstyle{plain}
\newtheorem{theorem}{Theorem}
\newtheorem{proposition}[theorem]{Proposition}

\theoremstyle{definition}

\theoremstyle{remark}

\DeclareFontFamily{U}{bigshuffle}{}
\DeclareFontShape{U}{bigshuffle}{m}{n}{
  <5-8> s*[1.7] shuffle7
  <8->  s*[1.7] shuffle10
}{}
\DeclareSymbolFont{BigShuffle}{U}{bigshuffle}{m}{n}
\DeclareMathSymbol\bigshuffle{\mathop}{BigShuffle}{"001}
\DeclareMathSymbol\bigcshuffle{\mathop}{BigShuffle}{"002}

\begin{document}

\sloppy

\title{The Shortest Interesting Binary Words}

\date{}

\author{Gabriele Fici\\[1mm] Dipartimento di Matematica e Informatica, Universit\`a di Palermo, Italy\\[1mm] gabriele.fici@unipa.it}


\maketitle

\begin{abstract}

I will show that there exist two binary words (one of length 4 and one of length 6) that play a special role in many different problems in combinatorics on words. They can therefore be considered \textit{the shortest interesting binary words}. My claim is supported by the fact that these two words appear in dozens of papers in combinatorics on words.
\end{abstract}


\section{Introduction}

Many papers are devoted to the study of properties of some interesting infinite word, e.g., the Fibonacci word $f=0100101001001\cdots$, or the  Thue--Morse word $t=0110100110010110\cdots$, or to the study of classes of  words. But to the best of my knowledge no paper has been entirely devoted to just two short binary words!
In this paper, I focus on the  words: \[v=0011\]  and  \[w=001011.\]

Why do I claim that these two words are  interesting? An answer could be that they appear in no less than 50  papers in combinatorics on words. They are probably the shortest binary words that are not \emph{too} trivial. For this reason, they are often presented as an example for many classical definitions, e.g., primitive word, unbordered word, Lyndon word, Dyck word, etc. But, as it will be shown in this paper, they also have many other surprising properties. 

As usual in the field, I will use the last letters of the Latin alphabet to denote words, i.e., $u$, $v$, $w$, etc.
To convince the reader that these two words are of \emph{particular} relevance in the field, think of the diagonal lattice representation of a binary word, that is, the diagonal lattice path obtained encoding each $0$ with a downstep ($\backslash$), i.e., a segment that goes from a point $(i, j)$ to $(i +1, j - 1)$, and each $1$ with an upstep ($/$), i.e., a segment that goes from a point $(i, j)$ to $(i +1, j + 1)$. 

Then  the path encoding $v=0011$ has a \texttt{V} shape, and of course the path encoding $w=001011$ has a \texttt{W} shape!  \Laughey[1.0]

%
%

\section{Palindromes and Anti-palindromes}\label{sec:pal}

A first observation is that for both words $v=0011$ and $w=001011$ it holds that the mirror image ($\tilde{v}=1100$ and $\tilde{w}=110100$, respectively) has a different character in each position.   Words with this property are called \emph{anti-palindromes} (or \emph{sesqui-palindromes} \cite{DBLP:journals/ejc/CarpiL04}). 

But while for the word $v=0011$, the mirror image is a rotation (conjugate) of the word, this does not hold for the word $w=001011$, for which there is no rotation that yields the word $\tilde{w}=110100$.


 A word such that no two rotations coincide is called \emph{primitive}; a word such that no two rotations of the word or of its mirror image coincide, i.e., a word $u$ of length $n$ such that the set made by all rotations of $u$ and all rotations of $\tilde{u}$ has cardinality $2n$, is called \emph{asymmetric}~\cite{BrHaNiRe04}. 
No binary word of length smaller than $6$ is asymmetric, and there is a unique orbit of asymmetric binary words of length $6$, namely that of $w=001011$~\cite{BrHaNiRe04}.

The word $v=0011$ is not a palindrome, but can be written as a concatenation of two palindromes ($00$ and $11$). In general, when this happens, for a primitive nonempty word, the factorization is unique~\cite{DBLP:journals/tcs/LucaM94}. The word $w=001011$, instead, cannot be written as the concatenation of two palindromes. It is easy to see that a word $u$ is a concatenation of two palindromes (one of which could be empty) if and only if $u$ is a rotation  of its mirror image $\tilde{u}$.  


However, the word $w=001011$ can be written as the concatenation of 3 palindromes (for instance, $w=00\cdot 101 \cdot 1$; notice that, contrarily to the case of $2$ palindromes, there may exist different factorizations in more than $2$ palindromes). It therefore has (see~\cite{DBLP:journals/aam/FridPZ13}) \emph{palindromic length} 3, while the word $v=0011$ has palindromic length $2$. 
Actually, $w=001011$  is a binary word of minimal length having palindromic length 3, that is, which can  be written as a concatenation of  3, but not fewer, palindromes. It is minimal in the sense that all of its proper factors have palindromic length at most $2$~\cite{DBLP:journals/combinatorics/BorchertR15}.

The shortest binary word with palindromic length 4 is $w^{8/6}=00101100$, up to mirror image and character exchange; the shortest binary word with palindromic length 5 is $w^{11/6}=00101100101$, up to mirror image and character exchange.

The shortest binary word with palindromic length 6 has length 14, but it is no longer a fractional power of $w$ (it is in fact the word $00101110001011=w10w$, up to mirror image).

The word $w^{11/6}$ is in fact an exception for the sequence $P(n)$ of the maximum palindromic length a binary word of length $n$ can have, since Ravsky \cite{Ra03} showed that the sequence $P(n)$ is given by $$P(n)=\lfloor n/6 \rfloor + \lfloor (n+4)/6 \rfloor +1$$ for every $n\neq 11$, and $P(11)=5$.

\bigskip

Regarding the number of distinct palindromic factors, one has that $v$ has $5$ palindromic factors ($\epsilon$, $0$, $1$, $00$ and $11$) and $w$ has $7$ (the same of $v$ plus $010$ and $101$). It is well known (and indeed easy to prove, see~\cite{DBLP:journals/tcs/DroubayJP01}) that any word of length $n$ contains at most $n+1$ distinct palindromic factors, including the empty word $\epsilon$. A word of length $n$ containing $n+1$ distinct palindromic factors is called \emph{rich}, or \emph{full}. 

So the words $v$ and $w$ are both rich. 

Actually, every binary word of length $7$, or less, is rich. 
The word $w^{8/6}=00101100$ is a non-rich binary word of minimal length. Indeed, it has length $8$ and only $8$ palindromic factors, namely $\varepsilon,0,1,00,11,010,101$, and $0110$.

It can be proved that a word is rich if and only if  all of its factors also are. Hence, it is natural to extend the definition to infinite words: An infinite word is called rich if all its finite factors are rich. For example, the Fibonacci word $f=0100101001001\cdots$ is rich. The word $w^{8/6}$ is a factor of the Thue--Morse word $t=0110100110010110\cdots$, so the Thue--Morse word is not rich.

One may wonder whether all binary palindromes are rich. This is not the case. An example (of minimal length) is the word  $00101100110100$ of length $14$,  the shortest palindrome that starts with $w^{8/6}$. 

A word $u$ is called \emph{circularly rich} if $u^2$ (or, equivalently, the infinite word $u^\infty=uuu\cdots$) is rich. Surprisingly, this is not equivalent to the fact that $u$ and all its rotations are rich. A counterexample is again provided by the word $w=001011$: all its rotations are rich but the word $w^2=001011001011$ is not rich, since it contains the non-rich factor $w^{8/6}=00101100$. Glen et al.~\cite{GlJuWiZa09} proved 
that a word $u$ is circularly rich if and only if $u$ and all its rotations  are rich \emph{and} $u$ is the concatenation of two palindromes.

For the same reason, the infinite word $w^\infty$ is not rich. Actually, it is the infinite binary word containing \emph{the least} number of palindromic factors! The set of palindromic factors of $w^\infty$ is $\{\epsilon, 0, 1, 00, 11, 010, 101, 0110, 1001\}$ and so has cardinality 9. It has been shown that every infinite binary word contains at least 9 distinct palindromic factors~\cite{FiZa13}. Moreover, an infinite binary word has exactly 9 distinct palindromic factors if and only if it is of the form $z^{\infty}$ where $z$ is a rotation of $w$ or a rotation of $\tilde{w}$.

An \emph{aperiodic} binary word, instead, must contain at least 11 distinct palindromic factors~\cite{FiZa13}. An example of such a word is the fixed point of the morphism $0\mapsto 0001011,\ 1\mapsto 001011$, i.e., $0\mapsto 0w, \ 1\mapsto w$. It is aperiodic for known properties of fixed points of binary morphisms~(see for example~\cite[Proposition 14]{DBLP:conf/dlt/FrosiniMRRS22}).

%

   \bigskip
   
A word $u$ is called a {\it palindromic periodicity} 
if there exist two palindromes $p$ and $s$ such
that $|u| \geq |ps|$ and
$u$ is a prefix of the word $(ps)^\omega = pspsps\cdots$~\cite{Simpson:2024}.  

No infinite binary word has fewer than 30 distinct palindromic periodicities.  The periodic word
$w^\infty$ has $30$~\cite{FSS:2024}. 
  
   \bigskip
   
A word is called \emph{weakly rich}~\cite{GlJuWiZa09} if the factor separating any two consecutive occurrences of the same character  is always a palindrome. It can be proved that all rich words are weakly rich, but the converse does not always hold. For example, the word $w^{8/6}=00101100$ is weakly rich (since, trivially, all binary words are weakly rich) but it is not rich. The word $0010200$ is a non-binary word that is weakly rich but not rich; the word $0120$ is not rich nor even weakly rich.
   
Every weakly rich word $u$ can be uniquely reconstructed (up to a permutation of characters) from the set 
\[S(u)=\{(i,j) \mid u[i..j]\mbox{ is a palindrome}\},\] 
since the pairs $(i,j)$ in $S(u)$ induce a set of equations that partitions $\{1,\ldots,|u|\}$ in subsets of positions containing the same character. To reconstruct the word, one assigns a different character to each part.

If a word $u$ is not  weakly rich, the information from the set $S(u)$ is not sufficient to uniquely reconstruct $u$. For example, for $u=0120$ and $z=0123$ one has $S(u)=S(z)=\{(1,1),(2,2),(3,3),(4,4)\}$.


The \emph{minimal palindromic specification} of a weakly rich word $u$ is the cardinality of a smallest subset $S'(u)$ of $S(u)$ that allows one to uniquely reconstruct $u$, i.e., that induces the same set of equations as $S(u)$ (cf.~\cite{DBLP:journals/jct/HarjuHZ15}). For example, words for which the minimal palindromic specification is equal to $1$ are $u=00$ ($S'(u)=\{(1,2)\}$), $u=010$ ($S'(u)=\{(1,3)\}$), $u=0110$ ($S'(u)=\{(1,4)\}$), and $u=01210$ ($S'(u)=\{(1,5)\}$; the minimal palindromic specification of $u=0^n$, $n\geq 3$, is $2$ ($S'(u)=\{(1,n),(2,n)\}$);  the minimal palindromic specification of $u$ is $0$ if and only if all characters in $u$ are distinct ($S'(u)=\emptyset$); finally, the minimal palindromic specification of $u=0100101001$ is $3$ ($S'(u)=\{(1,6),(4,6),(2,10)\}$.

Actually, if $u$ is any binary balanced word (see Section~\ref{sec:bal}), then the minimal palindromic specification of $u$ is at most $3$~\cite{DBLP:journals/jct/HarjuHZ15}.

A shortest word that has minimal palindromic specification equal to $4$ is the word $w=001011$. Indeed,  $w$ can be uniquely reconstructed from $S'(w)=\{(1, 2), (2, 4), (3, 5), (5,6)\}$ but not from any subset of $S(w)$ of cardinality less than $4$.
   

  \bigskip

The \emph{derivative} of a (finite or infinite) binary word  is the sequence of consecutive differences of characters (interpreted as integers) and is, in general, a ternary word. In this way, the characters of the derivative are in $\{-1,0,1\}$. In order to obtain a word over the alphabet $\{0,1,2\}$, one can add $1$ to each consecutive difference. Thus, in this paper I define the derivative of the word $u_1u_2\cdots$ as the word whose $i$th character is $1+u_i-u_{i+1}$. 

For example, it is well known the derivative of the Thue--Morse word $t=0110100110010110\cdots$ is a square-free ternary word, $0120210121020120210\cdots$; while the derivative of the Fibonacci word is the word $2012020120120201202012\cdots$ obtained by applying the morphism $0\mapsto 201,\ 1 \mapsto 20$ to the Fibonacci word.

In the case of finite words one has:

\begin{proposition}
 The derivative of an anti-palindrome is always a palindrome. On the opposite, the derivative of a palindrome is never an anti-palindrome.
\end{proposition}

\begin{proof}
 The first statement is evident by symmetry. For the second statement, observe that the derivative of a palindrome is either the word $1^n$, for some $n$, or a word that contains at least one occurrence of the character $2$.
\end{proof}

The derivative of $v=0011$ is $101$, a shortest palindrome containing $2$ different characters. The derivative of $w=001011$ is $10201$, a shortest palindrome containing $3$ different characters.

\begin{proposition}
 The derivative of $w^\infty=(001011)^{\infty}$, i.e., the word $(102012)^{\infty}$, is rich.
\end{proposition}

\begin{proof}
An infinite periodic word $u^\infty$ is rich if and only if $u^2$ is rich~\cite{GlJuWiZa09}.  The word $(102012)^2$ is rich. 
\end{proof}


Another related transformation is the \emph{Pansiot coding}~\cite{DBLP:journals/dam/Pansiot84} (also called \emph{Lempel homomorphism}~\cite{DBLP:journals/tc/Lempel70}) of a binary word, which consists in taking  the \emph{absolute value} of the consecutive differences (or, equivalently, the consecutive sums modulo $2$), and is therefore another binary word. For example, the Pansiot coding of the Thue--Morse word is the period-doubling word:  $10111010101110111\cdots$; while the Pansiot coding of the Fibonacci word is the word $110111101101111011110\cdots$ obtained by applying  the morphism $0\mapsto 11$, $1\mapsto 0$ to the Fibonacci word.

The Pansiot coding of the word $v=0011$ is the word $010$, while that of the word $w=001011$ is the word $01110$. 

Let me call a binary word $u$ a \emph{Pansiot pre-palindrome} if its Pansiot coding is  a palindrome.

\begin{proposition}
 A binary word $u$ is a Pansiot pre-palindrome if and only if $u$ is a palindrome or an antipalindrome. 
\end{proposition}

\begin{proof}
By induction on the length.
\end{proof}

Analogously, a word $u$ is a \emph{Pansiot pre-antipalindrome} if its Pansiot coding is  an antipalindrome.

Pansiot pre-antipalindromes can be generated recursively. Such words clearly have odd lengths, since antipalindromes must have even length and the Pansiot coding reduces the length by 1.

\begin{proposition}
The Pansiot pre-antipalindromes of length $3$ are: $001$, $011$, $100$, and $110$, and for every $n\geq 1$:
\begin{itemize}
  \item The Pansiot pre-antipalindromes of length $4n+1$ are precisely the words of the form $0u0$ or $1u1$, where $u$ is a Pansiot pre-antipalindrome of length $4n-1$.

   \item  The Pansiot pre-antipalindromes of length $4n+3$ are precisely the words of the form $0u1$ or $1u0$, where $u$ is a Pansiot pre-antipalindrome of length $4n+1$.

\end{itemize}
 \end{proposition}

\begin{proof}
By induction on $n$.
\end{proof}

\bigskip

By the way, a transformation that maps palindromes to anti-palindromes and anti-palindromes to palindromes exists~\cite{DBLP:journals/ejc/CarpiL04}: it is the Thue--Morse morphism $\tau:0\mapsto 01,\ 1 \mapsto 10$. So, for example, $\tau(v)=01011010$ and $\tau(w)=010110011010$ are indeed palindromes.


The Thue--Morse morphism has another fundamental property: Recall that an \emph{overlap} is a square followed by its first character, i.e., a word  of the form $auaua$, with $a$ a character and $u$ a word, e.g., $0010010$. A binary word $u$ is overlap-free (i.e., none of its factors is an overlap) if and only if $\tau(u)$ is overlap-free (see~\cite{Berstel}). Hence, for example, both $\tau(v)$ and $\tau(w)$ are overlap-free. 

More generally, Richomme and S{\'{e}}{\'{e}}bold  proved that a morphism $\mu$ is overlap-free, i.e., maps overlap-free words to overlap-free words, if and only if $\mu(w)=\mu(001011)$ is overlap-free~\cite{DBLP:journals/dam/RichommeS99}, and there is no shorter word that can replace $w$.

\bigskip

Consider now this problem: Given an integer $k>0$, is it possible to construct an infinite binary word that does not contain the mirror image of any of its factors of length $k$? 
Rampersad and Shallit~\cite{RaSh05} showed  that this is impossible for $k<5$: every binary word of length greater than $8$ contains at the mirror image of at least one factor of length $k$. But they proved that  the word $w^\infty=(001011)^{\infty}$ is an infinite binary word avoiding the mirror images of all its factors of length $\ge 5$. 

\bigskip

Currie and Lafrance~\cite{DBLP:journals/combinatorics/CurrieL16} showed that if one replaces  each $1$ by $w=001011$ in the Thue--Morse word, one obtains an  infinite binary word such that no factor is of the form $xyx\tilde{y}x$, for nonempty words $x$ and $y$.
If instead one replaces each $1$ by $w01111=00101101111$, one obtains an infinite binary word such that no factor is of the form $xyx\tilde{y}\tilde{x}$; while replacing each $1$ by $w11=00101111$ one obtains an infinite binary word such that no factor is of the form $xy\tilde{x}\tilde{y}x$~\cite{DBLP:journals/combinatorics/CurrieL16}.

\section{Squares and Other Repetitions}\label{sec:rep}

Every binary word of length at least $4$  contains  a square factor. Fraenkel and Simpson considered the problem of determining the largest number of square factors in a binary word~\cite{DBLP:journals/jct/FraenkelS98}. In order to count square factors, it is convenient to restrict the attention to primitive rooted squares (squares of the from $uu$ with $u$ a primitive word). The maximum number of distinct primitive rooted squares in a binary word of length $n$ is presented in Table II of~\cite{DBLP:journals/jct/FraenkelS98}. In particular, the word $v=0011$ is a word of minimal length containing 2 distinct primitive rooted squares ($00$ and $11$), while the word $w=001011$ is a word of minimal length containing 3 distinct primitive rooted squares ($00$, $11$, and $0101$).

So, infinite binary words cannot avoid squares. But there are infinite binary words avoiding overlaps. An example is the Thue--Morse word.

The word $w=001011$ is also an extremal case  for the following well-known result due to Restivo and Salemi~\cite{DBLP:conf/litp/RestivoS84}  (see also~\cite{DBLP:journals/combinatorics/AlloucheCS98}):

\begin{theorem}
 If a binary word $u$ is overlap-free,  then  there  exist $x,y,z$ with $x,z\in\{\varepsilon,0,1,00,11\}$ and $y$  overlap-free word,  such  that $u=x\tau(y)z$, where $\tau$ is the Thue--Morse morphism.  Furthermore this  decomposition is unique if $|u|\geq 7$, and $x$ (resp.~$z$) is completely determined by the prefix (resp.~suffix) of length $7$ of $u$.  The bound $7$ is sharp as shown by the example $w=001011 = 00\tau(1)11 = 0\tau(00)1$.
\end{theorem}

Currie and Rampersad~\cite{DBLP:journals/dmtcs/CurrieR10} proved that it is possible to construct an infinite binary word  avoiding cubes but containing exponentially many distinct square factors. 
They considered the (uniform) cube-free morphism
\begin{align*}
 0 & \mapsto 001011\\
 1 & \mapsto 001101\\
 2 & \mapsto 011001
\end{align*}

Notice that the morphism above maps  $0$ to $w$,  $2$ to a rotation of $w$, and  $1$ to a rotation of the complement\footnote{The complement of a binary word is the word obtained by exchanging $0$s and $1$s.} of $w$.

Recently, Dvořáková et al.~\cite{DOO24} proved that applying the $7$-uniform morphism
\begin{align*}
 0 & \mapsto 0001011\\
 1 & \mapsto 1001011
\end{align*}
that is, the morphism that maps $0$ to $0w$ and $1$ to $1w$, to  any  binary $\frac{7}{3}^+$-free  word  (i.e., a word such that no factor has exponent larger than $7/3$, where the exponent is defined as the ratio between the length and the minimum period) gives  a  cube-free  binary  word  containing  at  most  $13$ palindromes, which is the least number of distinct palindromes a binary cube-free word can have.

\bigskip

The word $v=0011$ is also an \emph{anti-square}, i.e., a word of the form $u\overline{u}$, where $\overline{u}$ is the complement of $u$; while $w=001011$ is not.
In particular, $v=0011$ is also a  \emph{minimal} anti-square, that is, an anti-square that does not properly contain any anti-square factor, except possibly $01$ and $10$. Minimal anti-squares have been characterized in \cite{anti}. In the same paper, the authors proved that a binary word that does not contain any anti-square factor, except possibly $01$ and $10$, and has length at least $8$, must contain $w=001011$, or its complement $\overline{w}$, as a factor. 

\bigskip

A different kind of repetition is the notion of a \textit{run} (or \textit{maximal repetition}). A pair $(i,j)$ is a run in a word $u=u[1..n]$, $1\leq i<j\leq n$, if the exponent of $u[i..j]$ is at least $2$ and is smaller than both the exponents of $u[i-1..j]$ and $u[i..j+1]$, if these are defined. For example, the runs of $w=001011$ are $(1,2)$, $(2,5)$ and $(5,6)$. Runs are particularly important in text processing, since they allow the design of efficient algorithms that process separately the repetitive and the non-repetitive portions of a string. 

It was conjectured in \cite{DBLP:conf/focs/KolpakovK99}, and then proved in \cite{DBLP:journals/siamcomp/BannaiIINTT17}, that the number of runs in a word of length $n$ is less than $n$. 

In \cite{DBLP:conf/focs/KolpakovK99}, the authors proved that the total sum of exponents of runs in a word $u$, noted $\sigma(u)$, is linear in the length of $u$. 
For example, the maximal value of $\sigma$ for a word of length $4$ is $4$, and this is realized by the word $v=0011$, which has two runs of exponent $2$, namely $(1,2)$ and $(3,4)$; while the maximal value of $\sigma$ for a word of length $6$ is $6$, and this is realized by the word $w=001011$, which has three runs of exponent $2$. 

However, one can have $\sigma(u)>|u|$ for larger values of $|u|$. For example, take the word $0010100101$, of length $10$. It has runs $(1,10)$, $(1,2)$, $(4,9)$, $(6,7)$, $(7,10)$, of exponent $2$; and $(2,6)$, of exponent $5/2$. So, $\sigma(u)=25/2$. No other word of length $10$ has a larger value of $\sigma$.

\bigskip

Recall that a \emph{Dyck word} is a binary word that, considering $0$ as a left parenthesis and $1$ as a right parenthesis, represents a string of balanced parentheses.

The words $v=0011$ and $w=001011$ are both Dyck words.

Consider the morphism $\mu$:
\begin{align*}
 0 & \mapsto 01\\
 1 & \mapsto 0011\\
 2 & \mapsto 001011
\end{align*}
i.e., the morphsim that maps $0$ to $01$, $1$ to $v$ and $2$ to $w$. Mol, Rampersad and Shallit \cite{DBLP:conf/cwords/MolRS23} proved that a binary word is an overlap-free Dyck word if and only if  it is of the form either $\mu(x)$ for a square-free word $x$ over $\{0,1,2\}$ that contains no $212$ or $20102$, or of the form $0\mu(x)1$, where $x$ is square-free word over $\{0,1,2\}$ that begins with $01$ and ends with $10$, and contains no $212$ or $20102$.

%


\bigskip

The  \emph{perfect shuffle} of two words of the same length  $x=x_1x_2\cdots x_n$ and $y=y_1y_2\cdots y_n$  is the word $x\shuffle y=x_1y_1x_2y_2\cdots x_ny_n$.

Guo, Shallit and Shur \cite{DBLP:journals/corr/GuoSS15} observed that a word is an antipalindrome if and only if it is of the form $\overline{x}\shuffle \tilde{x}$. For example, $v=0011=01\shuffle 01=\overline{10}\shuffle \widetilde{10}$ and $w=001011=011\shuffle 001=\overline{100}\shuffle \widetilde{100}$.

But the structure of the word $w=001011$ in terms of the perfect shuffle operator can be further specialized. In fact,  $w=001011$ satisfies the equation $xy=y\shuffle x$. Actually, it is the shortest word with two different characters doing so. 

The \emph{ordinary shuffle} of two words $x$ and $y$ is the set of words obtainable from merging the words $x$ and $y$ from left to right, but choosing the next
symbol arbitrarily from $x$ or $y$. More formally, the ordinary shuffle of $x$ and $y$ is the set
$x\bigshuffle y = \{z \mid z = x_1y_1 x_2y_2 \cdots x_ny_n \mbox{ for some } n \geq 1 \mbox{ and words } x_1, \ldots, x_n, y_1, \ldots, y_n \mbox{ such that } x = x_1 \cdots x_n \mbox{ and } y = y_1 \cdots y_n\}$.

A word that belongs to $x\bigshuffle x$ for some word $x$ is called a \emph{shuffle square}. Since $v\in 01\bigshuffle 01$, $v$ is a shuffle square; while $w$ is not a shuffle square. Actually, $v=0011$ is the shortest Dyck shuffle square. 

Deciding whether a binary word is a shuffle square is not an easy task. Indeed, Bulteau and Vialette~\cite{DBLP:journals/tcs/BulteauV20} proved that this problem is NP-hard. Recently, He et al.~\cite{DBLP:journals/ejc/HeHNT24} proved that for every $n\geq 3$, the number of binary shuffle squares of length $2n$ is strictly larger than $2n \choose n$.

Words belonging to $x\bigshuffle \tilde{x}$ for some word $x$, instead, are called \emph{reverse shuffle squares}. Henshall, Rampersad, and Shallit~\cite{DBLP:journals/eatcs/HenshallRS12} proved that binary reverse shuffle squares are precisely the binary \emph{abelian squares}, i.e., binary words of the form $uu'$ where $u'$ is an anagram of $u$. 

Neither $v=0011$ nor $w=001011$ is an abelian square. But there is a rotation of $v$ that is an abelian square ($0110$), while no rotation of $w$ is an abelian square. This is because, in general, one has the following property: a binary word has at least one rotation (including the word itself) that is an abelian square if and only if  it has an even number of $0$'s and  an even number of $1$'s (a word in which all letters occur an even number of times is sometimes called a \emph{tangram}).


\section{Lyndon and de Bruijn Words}\label{sec:Lyn}

A \emph{Lyndon word} is a primitive word that is lexicographically smaller than all its rotations (or, equivalently, lexicographically smaller than all its proper suffixes). Here I use the order $0<1$.

The words $v=0011$ and $w=001011$ are both Lyndon words.

Moreover, the word $v=0011$ is the shortest binary word that has $2$ different factorizations in two Lyndon words: $0\cdot 011$ and $001\cdot 1$; while the word $w=001011$ is the shortest binary word that has $3$ different factorizations in two Lyndon words: $0\cdot 01011$, $001\cdot 011$, and $00101\cdot 1$ (cf.~\cite{DBLP:journals/dm/Melancon00,DBLP:conf/soda/BassinoCN04}). 

In fact, one has:

\begin{proposition}\label{hmld}
 For every $n\geq 3$, the shortest binary word that has $n$ distinct factorizations in two Lyndon words is the word $00(10)^{n-2}11$, of length $2n$.
\end{proposition}

\begin{proof}[Sketch of proof]
 A Lyndon word of length $>1$ starts with $0$ and end with $1$. To have $n$ distinct factorizations in two Lyndon words, one needs at least $n$ occurrences of $10$.
\end{proof}


 The \emph{right standard factorization} of a Lyndon word $u$ of length at least $2$ is $u=st$, where $t$ is the lexicographically least proper suffix of $u$ (or, equivalently, the longest proper suffix of $u$ that is a Lyndon word).

For example, the right standard factorization of $v=0011$ is $0\cdot 011$, while that of $w=001011$ is $0\cdot 01011$.

Since the words $s$ and $t$ in the right standard factorization are always Lyndon words (this can be proved by exercise), applying the right standard factorization recursively until one gets words of length $1$ defines the so-called \emph{right Lyndon tree} of a word $u$, i.e., the binary tree whose root is the word $u$, the leaves  are single letters, and the children of a factor $u'$ of length greater than $1$ are the words in the right standard factorization of $u'$.  

There is also a \emph{left standard factorization} of a Lyndon word $u$ (a.k.a.~\emph{Viennot factorization}). It is the factorization $u=st$, where $s$ is the longest proper prefix of $u$ that is a Lyndon word (but in general $s$ is not the lexicographically least proper prefix of $u$, which is always a single letter!). 

The left and right standard factorizations are not the same, in general. For example, the left standard factorization of $v=0011$ is $001\cdot 1$. However, for some Lyndon words the right and the left standard factorizations can coincide yet their Lyndon trees are different (the left Lyndon tree is defined by applying recursively the left standard factorization). The class of binary words for which the right and the left Lyndon tree coincide is precisely the class of primitive lower Christoffel words (i.e., balanced Lyndon words, see below).

\bigskip

A \emph{de Bruijn word} of order $k$ is a  word such that all words of length $k$ occur exactly once in it as (cyclic) factors, i.e., as factors if one concatenates the de Bruijn word with its prefix of length $k-1$.  For example, the word $v=0011$ is a binary de Bruijn word of order $k=2$.
 
 The following famous result is due to Fredricksen and Maiorana~\cite{DBLP:journals/dm/FredricksenM78}:
 
\begin{theorem}\label{fm}
 The lexicographically least de Bruijn word of order $k$ can be obtained by concatenating in lexicographic order the Lyndon words of length dividing $k$.
\end{theorem}

So, the word $v=0011$ is the lexicographically least binary de Bruijn word of order $2$, since it  is the concatenation, in lexicographically order, of the Lyndon words of length dividing $2$, i.e., the words $0$, $1$ and $01$.

The lexicographically least binary de Bruijn word of order $3$ is the word $0w1=00010111$. Indeed, it  is the concatenation, in lexicographically order, of the Lyndon words of length dividing $3$, i.e., the words $0$, $001$, $011$ and $1$.

It is known that a binary de Bruijn word of order $k$ cannot be extended to a binary de Bruijn word of order $k+1$, but it can be extended to a binary de Bruijn word of order $k+2$~\cite{DBLP:journals/ipl/BecherH11}. For example, the word $v=0011$, of order $2$, can be extended to the de Bruijn word  $0011001011110100=v\cdot w\cdot \tilde{w}$ of order $4$.

\bigskip

There are several generalizations of de Bruijn words that have been proposed in the literature. One is the following: A generalized de Bruijn word of order $k$  is a  word such that all primitive  words of length $k$ occur exactly once in it  as  (cyclic) factors.

The following result, due to Au~\cite{DBLP:journals/dm/Au15}, is analogous to Theorem~\ref{fm}:

\begin{theorem}
The lexicographically least generalized de Bruijn word of order $k$ can be obtained by concatenating in lexicographically order the Lyndon words of length $k$.
\end{theorem}

According to the previous theorem, the word $w=001011$ is the lexicographically least generalized binary de Bruijn word of order $3$, since it is the concatenation of the binary Lyndon words of length $3$: $001$ and $011$. The reader can verify that $w$ contains every binary primitive word of length $3$ as a cyclic factor exactly once.

Another generalization of de Bruijn words has been proposed in \cite{DBLP:journals/iandc/GabricHS22}: a binary word of length $n$ is a generalized de Bruijn word if for all $0\leq i\leq n$, the number of cyclic factors of length $i$ is $\min(2^i,n)$. Clearly, when $n$ is a power of $2$, this definition coincides with that of ordinary de Bruijn word. For $n=6$ there are $3$ generalized de Bruijn words, namely $000111$, $w=001011$ and $\tilde{w}=110100$.


\bigskip 

The \emph{Burrows--Wheeler Transform} (BWT) of a word $u$ is the word obtained by concatenating the last characters of the rotations of $u$ sorted in lexicographic order. For example, if $u=0120$, the list of sorted rotations of $u$ is $\{0012,0120,1200,2001\}$, so the BWT of $u$ is $2001$. By definition, the BWT of $u$ is the same as the BWT of any rotation of $u$, so here I consider only the BWT of Lyndon words. 

The BWT of  $v=0011$ is $1010$, while the BWT of $w=001011$ is $101100$, which is a rotation of $w$. Actually, for each $n$, only a few binary Lyndon words of length $n$ (e.g., only 13 for length 20) are rotations of their BWT.

A general combinatorial characterization of binary Lyndon words that are rotations of their BWT is missing, although partial results have been obtained~\cite{DBLP:journals/fuin/MantaciRRRS17}. 

Given a word $u$, its \emph{standard permutation} $\pi_u$ is defined by: $\pi_u(i)<\pi_u(j)$ if $u_i < u_j$, or $u_i = u_j$ and $i < j$. For example, the standard permutation of $101100$ is $ \pi_{101100} =\bigl(\begin{smallmatrix}
    1 & 2 & 3 & 4 & 5 & 6 \\
    4 & 1 & 5 & 6 &  2  & 3
  \end{smallmatrix}\bigr)$.
A word $u$ is the BWT of some word if and only if its standard permutation is cyclic.

Higgins~\cite{DBLP:journals/tcs/Higgins12} observed that a word $u$ is the BWT of a binary de Bruijn word of order $k$ if and only if $\pi_u$ is cyclic and $u=\tau(z)$ for some word $z$ of length $2^{k-1}$, where $\tau$ is the Thue--Morse morphism. Indeed, in a binary de Bruijn word of order $k$, each factor $z$  of length $k-1$ occurs preceded by $0$ and $1$, so in the matrix of sorted rotations, the two consecutive rows starting with $z$ end with $0$ and $1$, hence the BWT of the de Bruijn word $u$ is a word in $\{01,10\}^+$.

For example, $1010=\tau(11)$, and $1010$ is the BWT of a de Bruijn word of order $2$ (namely the de Bruijn word $v=0011$ of order $2$).
The binary words of length $8$ whose standard permutation is cyclic and that are images under $\tau$ of words of length $4$ are $10011010=\tau(1011)$ and $10100110=\tau(1101)$, which are the BWTs, respectively, of the order $3$ de Bruijn words $0w1=00010111$ and $00011101$.

\section{Factors and Scattered Subwords}\label{sec:bal}

A binary word of length $n$ has at most $2^{k+1}-1+\binom{n-k+1}{2}$ distinct factors, where $k$ is  the unique integer such that $2^k+k-1 \leq n \leq 2^{k+1}+k$~\cite{DBLP:journals/gc/Shallit93}. 

Equivalently, the maximum number of distinct factors of a binary word of length $n$  is
\begin{equation}\label{eq:maxbinary}
d(n)=\sum_{i=0}^n\min(2^i,n-i+1)
\end{equation}
and for each $n$ there are binary words realizing this bound.

\begin{table}[ht]
\begin{center}
\begin{tabular}{l!{\hspace{0.1em}}*{21}{wr{1.75em}}}
$n$ & 1 & 2& 3& 4& 5& 6& 7& 8& 9& 10& 11 & 12 & 13 &14 &15
\\
\hline 
$d(n)$ & $2$& $4$& $6$& $9$& $13$& $17$& $22$& $28$& $35$& $43$ & $51$ & $60$ & $70$ & $81$  & $93$ 
\\[2mm]
\end{tabular}
\caption{The maximum number of distinct factors of a word of length $n$.}
\end{center}
\end{table}

The words $v=0011$ and $w=001011$ both have the maximum number of distinct factors a word of the same length can have. The word $v$ has $9$ distinct factors: $\epsilon$, $0$, $00$, $001$, $0011$, $01$, $011$, $1$, and $11$; while $w$ has $17$ distinct factors:  $\epsilon$, $0$, $00$, $001$, $0010$, $00101$, $001011$, $01$, $010$, $0101$, $01011$, $011$, $1$, $10$, $101$, $1011$, and $11$.
 
\bigskip

Given a word $u$ of length $n$, a set $A\subseteq \{1,\ldots,n\}$ is an \emph{attractor} for $u$ if every factor of $u$ has at least one occurrence in $u$ crossing a position in $A$~\cite{DBLP:journals/corr/abs-1709-05314}. For example, $A=\{2,3\}$ is an attractor of $v=0011$. Moreover, $A$ is minimal, since if a word contains two different characters, all its attractors must have cardinality at least $2$. The shortest binary word having no attractor of size $2$ is $w=001011$, up to mirror image. 


\bigskip

A factor $x$ of a word $u$ is \emph{left} (resp.~\emph{right}) \emph{special} if $0x$ and $1x$ (resp.~$x0$ and $x1$) are factors of $u$; it is \emph{bispecial} if it is both left and right special.

For example, the only bispecial factor of $v=0011$ is $\epsilon$; while  the bispecial factors of $w=001011$ are: $\epsilon$, $0$, $01$, and $1$.

A word of length $n$ can have at most $n-2$ distinct bispecial factors~\cite{Del99,DBLP:journals/tcs/CarpiL01,DBLP:journals/ipl/Raffinot01}. 

Let me call a binary word of length $n$ \emph{highly bispecial} if it has the maximum number of distinct bispecial factors among the binary words of length $n$.

For example, for every $n\geq 4$, words of length $n$ with exactly $n-2$ distinct bispecial factors are $001^{n-3}0$ and $01^{n-2}0$.


\begin{proposition}
 For every $n\geq 4$, the lexicographically smallest highly bispecial word of length $n$ is $001^{n-3}0$, with the exception of $n=6$, for which the lexicographically smallest highly bispecial word is $w=001011$.
\end{proposition}

\begin{proof}[Sketch of proof]
For every $n\neq 6$, there are exactly three binary highly bispecial words of length $n$, namely $aab^{n-3}a$, $ab^{n-3}aa$, and $ab^{n-2}a$. For $n=6$, we have the same words plus $w$.
\end{proof}

\bigskip

A word $x$ is a \emph{minimal forbidden factor} (a.k.a.~\emph{minimal forbidden word} or \emph{minimal absent word}) of a word $u$ if $x$ is not a factor of $u$ but every proper factor of $x$ is. For example, $v=0011$ is a minimal forbidden factor of $w=001011$.
The set of minimal forbidden factors of a word uniquely characterizes it.

The minimal forbidden factors of $v=0011$ are $000$, $10$, and $111$. The word $w=001011$, instead, has 6 minimal forbidden factors: $000$, $0011$, $100$, $1010$, $110$, and $111$;  this is actually the maximum number of minimal forbidden factors a binary word of length $6$ can have. In fact, a binary word of length $n>2$ has at most $n$ distinct  minimal forbidden factors~\cite{DBLP:journals/tcs/MignosiRS02}. 

One may wonder whether the minimal forbidden factors of a palindrome are always palindromes. The answer is no: the shortest palindrome starting with $w=001011$, that is, the word $0010110100$ (see~\cite[Sec.~2.3]{DBLP:conf/cai/Berstel07}), is a palindrome of minimal length having a minimal forbidden factor that is not a palindrome, namely the word $v=0011$ (and its mirror image).

\bigskip

A binary word $u$ is \emph{balanced} if for every pair of factors $x,y$ of $u$ of the same length, the occurrences of $0$ (or, equivalently, of $1$) in $x$ and $y$ differ by at most one. Balanced binary words are precisely the finite factors of Sturmian words (Sturmian words are infinite words with $n+1$ distinct factors of length $n$ for every $n\geq 0$; for more on Sturmian words see, e.g.,~\cite{LothaireAlg}).

Every balanced word is rich~\cite{DelGlZa08}.

A binary word $u$ is unbalanced (i.e., not balanced) if and only if there exists a palindrome $z$ such that $0z0$ and $1z1$ are both factors of $u$~\cite{LothaireAlg}.
So, both $v=0011$ and $w=001011$ are unbalanced (taking $z=\varepsilon$). Therefore, they cannot appear as factors in any Sturmian word. Actually, the shortest unbalanced words are $v$ and its mirror image.

A  \emph{minimal unbalanced word} is an unbalanced word such that all its proper factors are balanced.
The words $v=0011$ and $w=001011$ are minimal unbalanced words.
Minimal unbalanced words have been characterized~\cite{DBLP:journals/jcss/Fici14}:

\begin{proposition}
 A word $u=azb$, $\{a,b\}=\{0,1\}$, is a minimal unbalanced word  if and only if the word $bza$ is a proper power of a Lyndon balanced word or its mirror image.
\end{proposition}

Since $01$ is a Lyndon balanced word, any word of the form $1(10)^{n-1}0$, $n> 1$, or its mirror image $0(01)^{n-1}1$ is a minimal unbalanced word -- and in this latter case one has the same words of Proposition~\ref{hmld}, since $0(01)^{n-1}1=00(10)^{n-2}11$. 

Lyndon balanced words are also called \emph{lower Christoffel words}. A fundamental property of Lyndon words is the following: if $u$ and $z$ are Lyndon words and $u<z$ (where $<$ denotes the lexicographic order induced by $0<1$) then $uz$ is a Lyndon word. This property is not preserved in the (sub)class of lower Christoffel words, as the following example (see~\cite{LapointePhd}) shows:  $001$ and $011$ are lower Christoffel words, but $001\cdot 011=w$ is not, since it is not balanced. In fact, Borel and Laubie~\cite{JTNB_1993} proved that if $u$ and $z$ are lower Christoffel words and $u<z$, then $uz$ is a lower Christoffel word if and only if 
\[\det
\begin{pmatrix}
|u|_0 &\ |z|_0 \\
|u|_1 &\ |z|_1 
\end{pmatrix}=1.
\]


\bigskip

Let $u$ be a  word of length $n$. A \emph{scattered subword} of length $l$ of
$u$ is any word obtained by concatenating the characters appearing in $l$ distinct positions (even not contiguous). The set of these $l$ positions is called an \emph{embedding} of the scattered subword. For example, 
the scattered subwords of length $4$ of $w=001011$ are $0001$, $0011$, $0111$, and $1011$. The word $v=0011$ is a scattered subword of the word $w=001011$ (this is precisely the example given in~\cite{DBLP:journals/tcs/Metivier85}) and has 5 embeddings in $w$ (see Fig.~1 in\cite{DBLP:conf/lata/BoassonC15}). 

Clearly, every binary word  contains a palindromic scattered subword of length at least half of
its length -- a power of the prevalent character. A  word is called \emph{minimal palindromic} if it contains no palindromic scattered subword longer than half of its length. 
Holub and Saari~\cite{DBLP:journals/dam/HolubS09} proved that minimal palindromic binary words are abelian unbordered, i.e., no prefix has the same number of $0$s as the suffix of the same length.  
The words $v=0011$ and $w=001011$ are minimal palindromic words (and therefore are abelian unbordered). 












\bibliographystyle{abbrv}

\end{document}